\documentclass[oneside,12pt]{amsart}
\usepackage{amssymb, amscd}
\usepackage[all]{xy}
\usepackage{dsfont}

\usepackage{a4wide}
\usepackage[english]{babel}
\usepackage{amsfonts}
\usepackage{amsmath}
\usepackage{amsthm}
\newtheorem{theorem}{Theorem}[section]
\newtheorem{definition}[theorem]{Definition}
\newtheorem{proposition}[theorem]{Proposition}
\newtheorem{lm}[theorem]{Lemma}
\newtheorem{corollary}[theorem]{Corollary}
\newtheorem{conjecture}[theorem]{Conjecture}

\newtheorem{bieberbach}[theorem]{Bieberbach}
\newtheorem{ex}[theorem]{Example}
\newtheorem{remark}[theorem]{Remark}

\title{Fiberwise volume decreasing diffeomorphisms on product manifolds}
\author{Dennis Dreesen}
\address{KULAK \\ Wiskunde \\ E. Sabbelaan 53, B-8500 Kortrijk}
\email{Dennis.Dreesen@kuleuven-kortrijk.be}
\thanks{Dennis Dreesen is a research assistant for FWO-Flanders}
\author{Nansen Petrosyan}
\address{KULAK \\ Wiskunde \\ E. Sabbelaan 53, B-8500 Kortrijk}
\email{Nansen.Petrosyan@kuleuven-kortrijk.be}
\thanks{Nansen Petrosyan was supported by the Research Fund K.U.Leuven}
\DeclareMathOperator{\vol}{Vol}
\DeclareMathOperator{\fvdc}{FVD}
\DeclareMathOperator{\Diffeo}{Diffeo}
\DeclareMathOperator{\Iso}{Iso}
\newcommand{\R}{\mathbb{R}}
\newcommand{\N}{\mathbb{N}}
\newcommand{\Q}{\mathbb{Q}}
\newcommand{\Z}{\mathbb{Z}}

\newcommand{\cmd}{\backslash}
\newcommand{\Diff}{\mbox{Diffeo}}
\newcommand{\aste}{\textasteriskcentered}

\begin{document}
\maketitle
\begin{abstract}
Given a closed connected Riemannian manifold $M$
and a connected Riemannian manifold $N$, we study fiberwise, i.e. $M\times \{z\}, z\in N$,
volume decreasing diffeomorphisms on the product $M\times N$. Our main theorem shows that
 in the presence of certain cohomological condition on $M$ and $N$
such diffeomorphisms must map a fiber diffeomorphically onto another fiber and are therefore
fiberwise volume preserving. As a first corollary, we show that
the isometries of $M\times N$ split. We also study
properly discontinuous actions of a discrete group on $M\times N$. In this case, 
we generalize the first Bieberbach theorem and prove a special case of an extension of Talelli's conjecture.
\end{abstract}
\maketitle
\section{Introduction \label{sc:introduction}}
When one studies a problem on a product manifold $M\times N$, then
it is convenient if this problem ``reduces'' to
two separate problems concerning $M$ and $N$.
We have noticed this in an attempt to generalize the first Bieberbach theorem to a case related to product manifolds.
Here, we needed the fact that under certain conditions the isometries of $M\times N$ split. An isometry is said to {\em split}
if its $M$-component $M\times N \rightarrow M$ is
independent of the $N$-coordinates and its $N$-component $M\times N \rightarrow N$
is independent of the $M$-coordinates. The component mappings can then be seen as isometries of $M$ and $N$, respectively.
In this article, we find conditions on $M$ and $N$ that admit such a splitting of isometries and thus allow for the reduction
of geometric symmetries of $M\times N$ to the components, $M$ and $N$.
At least one intermediate result is useful on its own and will be referred to as our main theorem.


The structure of the article is as follows.
In section \ref{se:preliminaries}, we recall definitions and state the preliminary results.
Our main theorem is proven in section \ref{sc:main}.
Section \ref{sc:corollaries} is concerned with applications of this theorem to properly discontinuous actions.
Throughout the article, $M$ and $N$ are Hausdorff and second countable Riemannian manifolds. Also,
we will only deal with smooth maps,
so differentiable means $C^{\infty}$. We denote $n=\dim(M)$. Let us formulate our results.

If $z\in N$, then $M\times \{z\} \subset M\times N$
is a manifold isometric to $M$ and $\vol(M\times \{z\})=\vol (M)$
(see Definition \ref{definition:nicefamily}).
Let $f:M\times N \rightarrow M\times N$ be a diffeomorphism. We say that $f$ is {\em fiberwise volume
decreasing} at $z$ if
\[ \vol(f(M\times \{z\})) \leq \vol(M),\]
where $f(M\times \{z\})$ has the induced metric from $M\times N$.
A diffeomorphism is fiberwise volume decreasing (fvd) if it is fiberwise volume decreasing at each $z \in N$. Note that isometries
are fvd.

We say that a pair $(M,N)$ satisfies the {\em \aste-condition} if
\[ \pi^*:H^n(M;\Z_2)\rightarrow H^n(M\times N;\Z_2) \]
is an isomorphism. Here, $\pi:M\times N \rightarrow M$ is the natural projection map. We obtain the following result.
\begin{theorem}[(Main Theorem)] Let $M$ be a closed connected Riemannian manifold and let $N$ be a Riemannian manifold
such that $(M,N)$ satisfies the \aste-condition. If $f:M\times N \rightarrow M\times N$ is a diffeomorphism which is fvd at $z\in N$,
then there exists $w\in N$ such that $f(M\times \{z\})=M\times \{w\}$.\label{th:mainin}
\end{theorem}

Under the assumptions of the theorem, one can show that the collection of fiberwise volume decreasing self-maps
of $M\times N$, denoted by $\fvdc(M\times N)$, is a group. To describe its structure, we
define the following map $\psi$ whose definition does not depend on the chosen $y_0\in M$.
\[ \begin{array}{cccl}
\psi: & \fvdc(M\times N) &\rightarrow & \Diffeo(N) \\
 & g &\mapsto & \widetilde{p\circ g} ,
 \end{array} \]
where $p:M\times N \rightarrow N$ is the natural projection and
\[ \begin{array}{cccl}
\widetilde{p\circ g}: & N & \rightarrow & N \\
&  z& \mapsto & p\circ g(y_0,z).
\end{array} \]
For every element of $K:=\{ f:N\rightarrow \Diffeo(M) \mid f \mbox{ is Fr\'echet differentiable} \}$, we can associate the map
\[
\begin{array}{cccl}
\overline{f}: & M\times N & \rightarrow & M \\
 & (y,z)& \mapsto & f(z)(y)
 \end{array}
 \]
and so $(\overline{f},p)\in \fvdc(M\times N)$.
We obtain the following exact sequence.
\begin{theorem}
The sequence:
\[ 1 \rightarrow K \hookrightarrow \fvdc(M\times N) \stackrel{\psi}{\rightarrow}
\Diff(N) \rightarrow 1 \] 
is short exact.
\end{theorem}

An important application of our main theorem is related to the splitting of isometries of a product manifold.
Cheeger and Gromoll (see \cite{Ch}) have shown that the isometries of $M\times \R^k$ split for any $k\in \N$.
They use the fact that any point of $\R^k$ lies on a line through a given point to eventually show that
isometries map fibers of the form $M\times \{z\}$ to fibers of the same form.
A line in a complete Riemannian manifold $N$
is a geodesic $\gamma :(-\infty, \infty) \to N$ that minimizes the arc length between any two of its points.
The $3$-dimensional Heisenberg group shows that not
even contractible Lie groups with a left-invariant metric need to satisfy this property (see \cite{mar}).
Using Theorem \ref{th:mainin}, we are able to avoid these complications.

\begin{corollary} [(Splitting Theorem)] If $M$ is a closed connected Riemannian manifold and if $N$ is a
connected Riemannian manifold such that $(M,N)$ satisfies
the \aste-condition, then the isometries of $M\times N$ split, i.e. $\mbox{Iso}(M\times N)=\mbox{Iso}(M)\times \mbox{Iso}(N)$. \label{cor:split}
\end{corollary}

It is worth noting that for a complete $N$ the theorem follows by the de Rham decomposition (see \cite{EH}).
Throughout the article we make no completeness assumption on $N$, except in theorem \ref{theorem:bin}.

An interesting application of our splitting theorem is related to groups that can act
properly discontinuously, cocompactly and isometrically on products $M\times N$.
When $M$ is a singleton and $N$ is a simply connected, connected, nilpotent Lie group,
equipped with a left-invariant metric, such groups are called {\em almost-crystallographic groups}.
All of the three famous Bieberbach theorems (see \cite{BB1}, \cite{BB2} and \cite{BB3})
have been generalized to almost-crystallographic groups.
The following generalization of the first Bieberbach theorem was given by L. Auslander.

\begin{bieberbach}
[(Generalized first Bieberbach theorem, Auslander, \cite{Auslander})] Every almost-crystallographic
group contains a finite index subgroup isomorphic to a uniform lattice, i.e.\! a discrete and cocompact subgroup of $N$.
\label{bb:Auslander}
\end{bieberbach}

We obtain the following generalization.

\begin{theorem}
Let $M$ be a closed connected Riemannian manifold and let $N$ be a simply
connected, connected, nilpotent Lie group equipped with a left-invariant metric.
If $\Gamma$ is a group acting properly discontinuously, cocompactly and isometrically on $M\times N$,
then $\Gamma$ contains a finite index subgroup isomorphic to a uniform lattice of $N$. \label{theorem:bin}
\end{theorem}

Since $N$ is a Lie group with a left-invariant metric, it must be complete. Then, the de Rham decomposition gives an alternate
proof of this result (see Remark \ref{rem:new}, Section \ref{sc:corollaries} ). We note that
one can easily find examples showing that the other Bieberbach-theorems do not generalize to the $M\times N$ case.

Another interesting setting for applying Theorem \ref{th:mainin} is
Talelli's conjecture (Conjecture III of \cite{tal}).
Let us denote the cohomological dimension of a group $\Gamma$ by $cd(\Gamma)$.
We study the following, slightly different version of the conjecture (see \cite{tal}).

{\conjecture [(Talelli conjecture reformulated, $2005$)] If $\Gamma$ is a torsion-free group that acts smoothly and
properly discontinuously on $S^n \times \R^k$, then $cd(\Gamma)\leq k$.}

By a result of Mislin and Talelli (\cite{tal:2}),
we know that the conjecture holds for the large class of $LH\mathcal{F}$-groups (see \cite{Kropholler}).

In the context of this article it feels natural to replace $S^n$ by any closed, connected Riemannian manifold $M$, and to replace
$\R^k$ by any contractible Riemannian manifold $N$. By doing this, we obtain the following

{\conjecture [(Petrosyan, 2007)] If $\Gamma$ is a
torsion-free group acting smoothly and properly discontinuously on $M \times N$, then $cd(\Gamma)\leq \dim(N)$.}

Petrosyan has proven this conjecture in the case of $H\mathcal{F}$-groups and when $N$ is $1$-dimensional (see \cite{Petrosyan}).
We prove the following

\begin{theorem}
Let $M$ be a closed and connected Riemannian manifold and
let $N$ be a contractible Riemannian manifold. If $\Gamma$ is a torsion-free group acting
properly discontinuously and fiberwise volume decreasingly on $M\times N$,
then $\Gamma$ acts freely and properly discontinuously on $N$. In particular, we have that $cd(\Gamma)\leq \dim(N)$.
\end{theorem}
\section{Background and preliminary results \label{se:preliminaries}}

Let $M$ be a Riemannian manifold of dimension $n$ and let $x: U \rightarrow M, U\subset \R^n$ be a parametrization of $M$.
For $i,j\in \{1,2,\ldots ,n\}$, vector fields $X_i$ and functions $g_{ij}$ on $x(U)$ are defined as follows:
let $p\in M$ and $q=(q_1,q_2,\ldots ,q_n) \in U$ such that $x(q)=p$. For each
$i\in \{1,2,\ldots, n\}$, consider the curve $x_i(t):=x(q_1,q_2,\ldots, q_i+t,q_{i+1},\ldots ,q_n)$ in $M$.
We define $X_i(p)= \frac{d}{dt} x_i(t) _{\mid t=0}$ and $g_{ij}(p)= \langle X_i(p),X_j(p) \rangle _p$.
The $g_{ij}$ are called the components of the metric tensor relative to the parametrization $x$.
To simplify notation, we will sometimes denote $g_{ij}(x(q))$ by $g_{ij}(q)$.

In section \ref{sc:main}, we will need the notions of measure $0$ and of volume of subsets of $M$. We define these here.

\begin{definition}
A subset $A$ of a manifold $M$ has {\it measure $0$} if $x^{-1}(A)$ has Lebesgue measure $0$ in $\R^n$
for every parametrization $x$ 
of $M$.
\end{definition}

Observe that the notion of measure $0$ is invariant under diffeomorphisms.
\begin{definition} Assume that $x:U\rightarrow M$ is a parametrization.
If $C$ is an open connected set such that $\overline{C}\subset x(U)$ is compact
and such that the boundary $\partial (C)$ of $C$ has measure $0$, then we call $C$ a {\em nice open of $M$}.
\label{definition:niceopen}
\end{definition}
The volume of a nice open $C$ of $M$ is defined by
\[ \vol(C)= \int\limits_{x^{-1}(C)} \! \sqrt{\det (g_{ij})} \, d\mu, \]
where the $g_{ij}$ are the components of the metric tensor relative to $x$ and where $\mu$ is the Lebesgue measure on $\R^n$.
The definition is independent of the parametrization used.
\begin{definition}
A diffeomorphism $f:M \rightarrow M$ is {\em volume preserving} if it preserves the volume of all
nice opens of $M$.
\end{definition}
\begin{definition}
Take a countable number of nice opens, say $(C_i)_{i\in I}$ where $I$ is some index set,
such that the $C_i$ are pairwise disjoint and such that
$M\cmd \bigcup_{i\in I} C_i$
has measure $0$. We call such a family a {\em nice family} for $M$.
We define the volume of $M$ as
\[ \vol(M)= \sum_{i\in I} \vol(C_i) .\] \label{definition:nicefamily}
\end{definition}
Let us elaborate on this definition. First of all,
note that this definition of volume is independent of the nice family chosen.

Secondly, there is a standard way of finding a nice family $(C_i)_{i\in I}$ for $M$.
Start with a countable number of parameterizations $x_1,x_2,\ldots $ whose images contain
the closures of nice open sets $B_1,B_2,\ldots$ respectively.
Make sure that $\cup_{i=1}^{\infty} B_i \supset M$.
Consider the sets
\[ B_1':=B_1,\, B_2':=B_2\cmd \overline{B_1},\, B_3':=B_3\cmd \overline{B_1\cup B_2},\, \ldots , \, B_n':=B_n\cmd \overline{\cup_{i=1}^{n-1} B_i}, \ldots \]
You can take the family $(C_i)_{i\in I}$ as
the family of connected components of the sets $B_i'$.


Finally, note that volume preserving diffeomorphisms preserve $\vol(M)$.\\

An important class of volume preserving diffeomorphisms is the class of isometries of a Riemannian manifold.
We will be primarily interested in isometries
of a product of manifolds $M$ and $N$. This product is again
a Riemannian manifold with inner product given by
\[ \langle v_1 \oplus w_1, v_2 \oplus w_2\rangle_{(y,z)} = \langle v_1,v_2\rangle _y + \langle w_1,w_2 \rangle_z,\]
for all $(y,z)\in M\times N, \, v_1,v_2 \in T_y(M)$ and $w_1, w_2\in T_z(N)$.
\begin{definition}
An isometry on a product of manifolds is said to {\em split}
if its $M$-component $f_1:M\times N \rightarrow M$ is
independent of the $N$-coordinates and its $N$-component $f_2:M\times N \rightarrow N$
is independent of its $M$-coordinates. In this case, the component mappings $f_1$ and $f_2$ can be seen as
isometries of $M$ and $N$ respectively.
\end{definition}

Note that all isometries of $M\times N$ split if and only if $\Iso(M\times N)=\Iso(M)\times \Iso(N)$.\\

\noindent The following theorem is a standard result from algebraic topology.
\begin{theorem} [(Poincar\'e-Lefschetz Duality, \cite{bredon})] Let $M$ be a compact orientable
$n$-manifold and let $L$ be a closed subset of $M$. Denoting \v{C}ech cohomology by $\check{H}$,
we have the following
commutative diagram where the rows are exact and all the vertical arrows
(cap products with the orientation class) are isomorphisms:
\[
\begin{array}{ccccccccccc}
\cdots & \rightarrow & \check{H}^p(M,L) & \rightarrow & \check{H}^p(M) & \rightarrow & \check{H}^p(L) &\rightarrow &
\check{H}^{p+1}(M,L) &\rightarrow & \cdots \\
& & \downarrow \approx  & & \downarrow \approx  & & \downarrow \approx  & & \downarrow \approx  & & \\
\cdots & \rightarrow & H_{n-p}(M\cmd L) & \rightarrow & H_{n-p}(M) & \rightarrow & H_{n-p}(M,M\cmd L) &\rightarrow &H_{n-p-1}(M\cmd L) &\rightarrow & \cdots
\end{array}
\]
For non-orientable $M$ the theorem holds with $\Z_2$-coefficients.\label{th:duality}
\end{theorem}

\noindent One obtains the following interesting
\begin{corollary}
If L is a proper closed subset of a compact $n$-manifold $M$, then
\[ \mid \check{H}^n(L;\Z_2) \mid < \mid \check{H}^n(M;\Z_2) \mid. \] \label{cor:sur}
\end{corollary}
\begin{proof}
The corollary follows from the fact that $\check{H}^n(L;\Z_2)$ is isomorphic to 
$H_0(M,M\cmd L;\Z_2)$ and this group contains less elements than $H_0(M;\Z_2)\cong \check{H}^n(M;\Z_2)$,
by Theorem \ref{th:duality}.
\end{proof}

\noindent We end this section by a purely algebraic lemma. Recall the following definitions.
\begin{definition}
A symmetric matrix $G$ in $\mathcal{M}_n(\R)$ is
positive definite if $x^TGx>0$ for every non-zero vector $x\in \R^k$.
A symmetric matrix $H\in \mathcal{M}_n(\R)$ is positive semi-definite if $x^THx\geq 0$ for every vector $x\in \R^k$.
\end{definition}
\begin{lm} If $G\in \mathcal{M}_n(\R)$ is positive definite and $H\in \mathcal{M}_n(\R)$ is positive semi-definite,
then $\det(G+H)\geq \det(G)$. The inequality is strict when $H\neq 0$.
\label{lm:posdef}
\end{lm}
\begin{proof}
We start by proving the special case where $H=E=(\mu,0,0,\ldots ,0)$ with $\mu\geq 0$. Here, the notation
$(e_{11},e_{22},\ldots ,e_{nn})$ stands for a diagonal matrix whose $(i,i)^{\mbox{th}}$ entry is $e_{ii}$.


\noindent Denote by $\widetilde{G}$ the matrix obtained from $G$ by removing the first row and column, i.e.
$\widetilde{G}_{ij}=G_{(i+1)(j+1)}$ for $i,j \in \{1,2,\ldots ,n-1\}$. Expanding $\det(G+E)$ by
the first row gives
\[ \det(G+E)=\det(G)+\mu \det(\widetilde{G}).\]
Since
\[ (x_1,x_2,\ldots ,x_{n-1})\widetilde{G} (x_1,x_2,\ldots x_{n-1})^T = (0,x_1,x_2,\ldots ,x_{n-1}) G (0,x_1,x_2,\ldots ,x_{n-1}), \]
for all $(x_1,x_2,\ldots ,x_{n-1})\in \R^{n-1}$, we have that $\widetilde{G}$ is positive definite.
This implies that $\det(\widetilde{G})>0$ 
and thus $\det(G+E)\geq \det(G)$.
Strict inequality
holds if and only if $\mu>0$. Notice that a similar proof exists
when $H$ equals a diagonal matrix of the form $(0,0,\ldots ,0,\mu,0,\ldots ,0)$.

In general,
take an orthogonal matrix $O$ such that $D=OHO^T=(\lambda_1,\lambda_2,\ldots ,\lambda_n)$. Clearly, $\lambda_i\geq 0$ for all $i$.
We have
\[ \det(G+H)=\det(OGO^T+D)=\det(OGO^T+E_1+E_2+\ldots +E_n),\]
where $E_i$ is the matrix that has $\lambda_i$ as its $(i,i)^{th}$ entry and zeros everywhere else.
By positive definiteness of $OGO^T$ we have
that $OGO^T+E_1+E_2+\ldots +E_k$ is positive definite for each $k\in \{1,2,\ldots ,n\}$.
The proof now follows from the special case proven above.
\end{proof}
\section{Main theorem \label{sc:main}}
\renewcommand{\thetheorem}{\thesubsection.\arabic{theorem}}
\subsection{Proof of the main theorem and splitting of isometries.}
From now on, assume that $M$ is an $n$-dimensional closed Riemannian manifold. Apart from being Riemannian, we put no conditions on
$N$. Consider the product manifold $M\times N$
and a point $z\in N$. Define the inclusion
\[
\begin{array}{cccl}
i: & M & \rightarrow & M\times N\\
 & y & \mapsto & (y,z)
 \end{array} \]
and the projection
\[ \begin{array}{cccl}
\pi: & M\times N &\rightarrow & M\\
 & (y,w) &\mapsto & y.
\end{array} \]
Clearly, the composition $\pi \circ i$ is the identity mapping of $M$ and so the mapping
\[ i^*\circ \pi^*: H^n(M;\Z_2) \mapsto H^n(M;\Z_2) \]
is an isomorphism.
Therefore, $\pi^*$ must be injective.

\begin{definition} If $M$ and $N$ are such that
\[ \pi^*:H^n(M;\Z_2)\rightarrow H^n(M\times N;\Z_2) \]
is an isomorphism or equivalently that
\[ i^* :H^n(M\times N;\Z_2)\rightarrow H^n(M;\Z_2) \]
is an isomorphism,
then we say that $(M,N)$ satisfies the {\em \aste-condition}. Note that this definition does not depend on the choice of $z\in N$.
\end{definition}

The following propositions will be useful in the proof of our main theorem.
\begin{proposition} If $(M,N)$ satisfies the \aste-condition, then
we have that
\[ \phi:= \pi\circ f\circ i :M\rightarrow M \]
is surjective for any homeomorphism $f:M\times N \rightarrow M\times N$.
 \label{proposition:surj}
\end{proposition}
\begin{proof}
Since $f$ is a homeomorphism and since $(M,N)$ satisfies the \aste-condition, we know that
\[ \phi^*: H^n(M;\Z_2) \stackrel{\pi^*}{\rightarrow} H^n(M\times N; \Z_2) \stackrel{f^*}{\rightarrow}
 H^n(M\times N; \Z_2) \stackrel{i^*}{\rightarrow} H^n(M;\Z_2) \]
is an isomorphism.
Assume now that $\phi$ is not surjective.
The image of $\phi$ is compact and thus closed. Since it misses a point, say $p$, it has to miss an open subset of $M$, say $U$.
Take a CW-complex structure on $M$ containing an open $n$-cell $\sigma$ with $p\in \sigma \subset U$.
Now, the forgetful map $\phi_1:M \rightarrow M\cmd \sigma$ of $\phi$ induces the mapping
$\phi^*_1:H^n(M\cmd \sigma;\Z_2)\rightarrow H^n(M;\Z_2)$.
Let $j$ be the inclusion mapping of $M\cmd \sigma$ into $M$. On the cohomology level we obtain
\[ \phi^*_1\circ j^*:H^n(M;\Z_2) \rightarrow H^n(M;\Z_2), \]
and this mapping equals $\phi^*$.
 Since $\phi^*$ is surjective,
we conclude that $\phi^*_1$ must be surjective, which is is a contradiction to Corollary \ref{cor:sur}
because \v{C}ech cohomology and singular cohomology are isomorphic for CW-complexes.
\end{proof}

\begin{proposition} Let
$f: M\times N \rightarrow M\times N$ be a diffeomorphism and let $z\in N$.
On $f(M\times \{z\})$, we consider the induced metric from $M\times N$.
Suppose that $C$ is a nice open in $M$ such that
the natural projection map $\pi:f(M\times \{z\})\rightarrow M$ restricts to a diffeomorphism $\phi$ onto an open set containing
$\overline{C}$.
Then, $\vol(\phi^{-1}(C))\geq \vol(C)$. Moreover, the equality is strict if and only if the projection $p:M\times N \rightarrow N$
is not constant on $\phi^{-1}(C)$. \label{pr:highervolume}
\end{proposition}

\begin{proof}
Let $x:U\rightarrow M$ be a parametrization for $M$ such that $\overline{C}\subset x(U)$.
Let $V=x^{-1}(C)$ and consider the parametrization
\[ \psi:=\phi^{-1}\circ x: V\rightarrow \phi^{-1}(C). \]
Write $\psi=(x,\eta)$ where $x:V\rightarrow M$ is the $M$-component map and where $\eta:V\rightarrow N$ is the $N$-component map
of $\psi$.
Denote the components of the metric tensor relative
to $x$ and $\psi$ by $g_{ij}$ and $\widetilde{g_{ij}}$ respectively.
By definition we have that
\[ \vol(C)= \int\limits_V \sqrt{\det(g_{ij})(q)} \,d\mu \]
and
\[ \vol(\phi^{-1}(C))= \int\limits_V \sqrt{\det(\widetilde{g_{ij}})(q)} \, d\mu. \]
To prove that $\vol(\phi^{-1}(C))\geq \vol(C)$ it thus suffices to
show that $\det(g_{ij}(q))\leq \det(\widetilde{g_{ij}}(q))$ for all $q\in V$.
Let us investigate the functions $\widetilde{g_{ij}}$.\\
For each $i\in \{1,2,\ldots ,n\}$ and $q=(q_1,q_2,\ldots ,q_n)\in V$, denote the curve
\[ x(q_1,q_2,\ldots , q_i+t, q_{i+1}, \ldots ,q_n) \]
by $x_i^q(t)$ and
\[ \psi (q_1,q_2,\ldots ,q_i+t, q_{i+1},\ldots q_n) \]
by $\psi_i^q(t)=(x_i^q(t),\eta_i^q(t))\in M\times N$.
For simplicity, we drop the upper index $q$ in the following calculation.
\begin{eqnarray*}
\widetilde{g_{ij}}(q) &=& \langle \psi_i'(0),\psi_j'(0) \rangle_{\psi(q)} \\
&=& \langle (x_i(t),\eta_i(t))'(0),(x_j(t),\eta_j(t))'(0)\rangle_{\psi(q)}\\
&=& \langle x_i'(0),x_j'(0)\rangle_{x(q)} + \langle \eta_i'(0),\eta_j'(0)\rangle _{\eta(q)}\\
&=& g_{ij}(q)+h_{ij}(q),
\end{eqnarray*}
where
\[ h_{ij}(q)= \langle (\eta_i^q)'(0), (\eta_j^q)'(0) \rangle_{\eta(q)}. \]

This shows that $\widetilde{g_{ij}}(q)=g_{ij}(q)+h_{ij}(q)$ for all $q\in V$. The first part of the
proposition now follows from Lemma \ref{lm:posdef}.

If $p\circ \phi^{-1}$ is
not constant on $C$, then $\vol(\phi^{-1}(C))>\vol(C)$ . Indeed, in this case there exists an open set $O\subset C$ such that the
linear map $D(p\circ \phi^{-1})_y \neq 0$ for
each $y\in O$. Let $W = x^{-1}(O)$.
We have that for each $q\in W$ there exists $i_q\in \{1,2,\ldots ,n\}$ such that
$(\eta_{i_q}^{q})' (0)=D(p\circ \phi^{-1})_{x(q)}((x_{i_q}^{q})'(0))\neq 0$.
The matrices $h_{ij}(q)$ are thus non-zero. Our claim
now follows from Lemma \ref{lm:posdef}.
\end{proof}

We give one more definition before proceeding with our main result.
\begin{definition} Let $f:M\times N \rightarrow M\times N$ be a diffeomorphism and let $z\in N$. Equip
both $M\times \{z \}$ and $f(M\times \{z \})$ with the Riemannian metric induced from $M\times N$ and note that
$\vol(M\times \{z\})=\vol(M)$.
We say that $f$ is {\em fiberwise volume decreasing at $z$} if
\[ \vol(f(M\times \{z \})) \leq \vol(M). \]
 A diffeomorphism is
{\em fiberwise volume decreasing} (fvd) if it is fiberwise volume decreasing at every point of $N$.
We denote the set of all fiberwise volume decreasing maps of $M\times N$ by $\fvdc(M\times N)$.
\end{definition}

\begin{theorem} Let $M$ be a closed connected Riemannian manifold and let $N$ be a Riemannian manifold
such that $(M,N)$ satisfies the \aste-condition.
If $f:M\times N \rightarrow M\times N$ is fvd at $z\in N$,
then there exists $w\in N$ such that $f(M\times \{z\})=M\times \{w\}$. \label{th:main}
\end{theorem}
\begin{proof} Assume that $f$ is fiberwise volume decreasing at $z$.
We prove the theorem by showing that \[ \vol(f(M\times \{z\} ))>\vol(M), \]
if $f(M\times \{z\})$ is not of the form $M\times \{w\}$ for some $w\in N$.
For the remainder of the proof we will denote $f(M\times \{z\})$ by $f(M)$.

Let $\pi$ be the natural projection map of $f(M)$ onto $M$.
From Proposition \ref{proposition:surj} it follows that $\pi \circ f_{\mid M\times \{z\}}$ is surjective.
Let's look at the set $A$ of critical values of $\pi$.
This set is closed and we know by Sard's theorem that it is of measure $0$ in $M$. Take a family of
nice opens $(C_i)_{i\geq 1}$ of $M$ that are pairwise disjoint, and such that their union equals $M\cmd \widetilde{A}$ where
$\widetilde{A}\supset A$ has measure $0$. We can assume this family to be such that the $C_i$ satisfy the hypotheses
of proposition \ref{pr:highervolume}.
We conclude that $\vol(f(M))\geq\vol(M)$.

Assume there exists a nice open $C\subset M$ such that
\begin{enumerate}
\item there are open subsets $V\subset f(M)$ and $O\subset M$ with
$\phi:=\pi_{\mid V}:V \rightarrow O$ a
diffeomorphism and $\overline{C}\subset O$ ,
\item $\vol(\phi^{-1}(C))>\vol(C)$.
\end{enumerate}
We can then look at a nice family of $M$ containing $C$ to conclude that $\vol(f(M))>\vol(M)$, obtaining the desired contradiction.
It remains thus to prove the existence of a nice open $C$, satisfying the two conditions
above, in case $f(M)$ is not a fiber.

Denote $p:f(M) \rightarrow N$ the projection map.
Assume by contradiction that for all $x \in f(M)$ the differential $(Dp)_{x}=0$ whenever $(D\pi)_{x}$ is an isomorphism, then
\[ \mathcal{A}_1=\{ x \in f(M) \mid Dp_{x} \neq 0 \} ,\]
and
\[ \mathcal{A}_2=\{ x \in f(M) \mid D\pi_{x} \mbox{ is an isomorphism} \} .\]
are disjoint, open, nonempty sets. Since $f$ is a diffeomorphism,
we have that $\mathcal{A}_1\cup \mathcal{A}_2=f(M)$. Since $M$ is connected, this is a contradiction. Hence, there exists
 an element $y \in f(M)$ such that $(Dp)_y \neq 0$ and $(D\pi)_{y}$ is an isomorphism.
Take a nice open $U\subset f(M)$ consisting of such points $y$. Let $u\in U$ with $\pi(u)\notin A$.
We can find a nice open $C\subset M\cmd A$ containing $\pi(u)$ that satisfies the hypotheses of Proposition \ref{pr:highervolume}.
Now, $\vol(\phi^{-1}(C))>\vol(C)$, as desired.
\end{proof}

{\remark {\rm For $(M,N)$ satisfying the \aste-condition, the proof can be generalized to the case that $M$ is not connected.
If $M_1,M_2,\ldots ,M_{k}$ are the connected components of $M$, and $(M,N)$ satisfies the \aste-condition,
then there exist $z_1,z_2,\ldots ,z_{k} \in N$ such that
\[ f(M\times \{z\}) = (M_1\times \{z_1\}) \cup (M_2\times \{z_2\}) \cup \ldots \cup (M_{k} \times \{z_{k}\}). \]
\begin{proof}
Proposition \ref{proposition:surj} does not use the fact that $M$ is connected
and so we know that $\pi\circ f_{\mid M\times \{z\}}$ is surjective. This implies that
$\pi \circ f$ maps each $M_i\times \{z\} $ surjectively onto an $M_j$.
The same reasoning as in the proof of Theorem \ref{th:main} then shows
that $\vol(f(M_i \times \{z\} )) \geq \vol(M_j)$. Since $f$ is fvd we can conclude that
\[ \sum_{l=1}^{k} \vol(f(M_l\times \{z\})) = \sum_{l=1}^{k} \vol(M_l), \]
and so $\vol(f(M_i\times \{z\}))=\vol(M_j)$. If we suppose that $f(M_i\times \{z\})$ is not of the form $M_j\times \{z_j\}$
for some $z_j\in N$, then, as in the proof of Theorem \ref{th:main}, using the connectedness of $f(M_i\times \{z\})$,
we can find a point $y\in f(M_i\times \{z\})$ such that $Dp_y\neq 0$ and $D\pi_y$ is an isomorphism. We can thus find a nice
open $C$ of $M_j\cmd A$ containing $\pi(y)$ that satisfies the hypotheses of Proposition \ref{pr:highervolume}. Therefore,
$\vol(\phi^{-1}(C))>\vol(C)$, implying $\vol(f(M_i \times \{z\} )) > \vol(M_j)$ and giving us a contradiction.
\end{proof} }}

We obtain the following interesting corollary. 
\begin{corollary} [(Splitting Theorem)] If $M$ is a closed connected Riemannian manifold and if $N$ is a
connected Riemannian manifold such that $(M,N)$ satisfies
the \aste-condition, then the isometries of $M\times N$ split, i.e. $\Iso(M\times N)=\Iso(M)\times \Iso(N)$. \label{cor:split}
\end{corollary}

\begin{proof}
Let $f=(f_1,f_2)$ be an isometry of $M\times N$. Then, $f$ satisfies the hypothesis of Theorem \ref{th:main}
and therefore $f_2$ is independent of its
$M$-coordinates. Notice that $f_2$ can thus be seen as a map from $N$ to $N$.

Let $(y,z)\in M\times N$ and denote $f_1(y,z)=x$.
A path $\gamma$ in $\{y\} \times N$, containing $(y,z)$, is
orthogonal to every fiber $M\times \{w\}$.
Since $f$ is an isometry which maps each fiber to another fiber, we have that $f\circ \gamma$
is orthogonal to each fiber $M\times \{w\}$. It is therefore a path in $\{ x \} \times N$ and connectedness of $N$
implies that
$f_1(\{y\} \times N)=\{ x \}$. Since $y\in M$ is arbitrary, we conclude
that $f_1$ does not depend on its $N$-coordinates. It can thus be seen as a map from $M$ to $M$.

\noindent Since $f$ is an isometry, we obtain that $f_1$ and $f_2$ are isometries of $M$ and $N$ respectively.
\end{proof}

\subsection{Fiberwise volume decreasing maps}
\setcounter{theorem}{0}
It is interesting to investigate which maps exactly are fvd. First of all,
we note that there
is no immediate connection with volume preserving maps. For example, on the cylinder $S^1\times \R\subset \R^3$ , one can
consider the diffeomorphism mapping $(\mbox{cos}(x), \mbox{sin}(x), y)\in S^1\times \R$ to $(\mbox{cos}(x), \mbox{sin}(x),{y\over 2})$. 
This map is clearly not volume preserving, but it
is fvd. Conversely, the diffeomorphism
\[
\begin{array}{cccc}
f: & S^1\times \R & \rightarrow & S^1\times \R \\
& (\mbox{cos}(x), \mbox{sin}(x),y) & \mapsto & (\mbox{cos}(x), \mbox{sin}(x),y+\sin(x))
\end{array}
\]
is volume preserving, since the Jacobian of the map $f$
 has determinant one at each point of $\R^2$. Yet, $f$ is not fvd.

Note further that $\fvdc(M\times N)$ has a natural group structure, because in our setting
``fiberwise volume decreasing'' and ``fiber preserving'' are equivalent notions.
Our main theorem implies the following

\begin{corollary} Given a point $y_0\in M$, consider
\[ \begin{array}{cccl}
\psi: & \fvdc(M\times N) &\rightarrow & \Diffeo(N) \\
 & (\alpha,\beta) &\mapsto & \widetilde{\beta} ,
 \end{array} \]
where
\[ \begin{array}{cccl}
\widetilde{\beta}: & N & \rightarrow & N \\
&  z& \mapsto & \beta(y_0,z).
\end{array}
\]
This definition is independent of the chosen $y_0$.
Furthermore, the map $\psi$ is a group homomorphism with kernel
\[ K= \{ f:N\rightarrow \Diffeo(M) \mid \overline{f}:M\times N\rightarrow M, (y,z)\mapsto f(z)(y) \mbox{ is differentiable.} \} \]
Additionally, there is a short exact sequence
\[ 1 \rightarrow K\cong \mbox{kernel}(\psi) \hookrightarrow \fvdc(M\times N) \stackrel{\psi}{\rightarrow}
\Diff(N) \rightarrow 1 .\] \label{cor:ses}
\end{corollary}

\begin{proof}
Theorem \ref{th:main} implies that the definition of $\psi$
is independent of the chosen $y_0\in M$.

To show that $\psi$ is a group homomorphism, let $(y,z)\in M\times N$ and
$(\alpha_1,\beta_1),(\alpha_2,\beta_2) \in \fvdc(M\times N)$. Then,
\[ (\alpha_1,\beta_1)\circ (\alpha_2,\beta_2)(y,z)= (\alpha_1(\alpha_2(y,z),\beta_2(y,z)),\beta_1(\alpha_2(y,z),\beta_2(y,z))) \]
and thus
\[ \psi((\alpha_1,\beta_1)\circ (\alpha_2,\beta_2))(z)=\beta_1(\alpha_2(y_0,z),\beta_2(y_0,z)). \]
On the other hand,
\[ \psi(\alpha_1,\beta_1) \circ \psi(\alpha_2,\beta_2)(z) = \beta_1(y_0,\beta_2(y_0,z)).\]
Both expressions are equal since $\beta_1$ doesn't depend on its first argument.

Observe that $\psi$ maps
each $(\alpha,\beta)\in \fvdc(M\times N)$ to a diffeomorphism of $N$.
This follows from the fact that $(\alpha,\beta)\in \fvdc(M\times N)$ has an inverse $(\alpha',\beta') \in \fvdc(M\times N)$
and so $\psi(\alpha',\beta')$ is an inverse for $\psi(\alpha,\beta)$. We conclude that $\psi$ is a well-defined group homomorphism.

Given a diffeomorphism $\gamma$ of $N$, define
\[ \begin{array}{cccl}
\hat{\gamma}: & M\times N & \rightarrow & N \\
& (y,z) &\rightarrow & \gamma(z).
\end{array}
\]
Let $\pi:M\times N\rightarrow M$ be the natural projection onto $M$.
Then, $(\pi,\hat{\gamma})\in \fvdc(M\times N)$ and $\psi(\pi,\hat{\gamma})=\gamma$. Hence, $\psi$
is surjective.

If $f$ is an element of $K$ and $p:M\times N \rightarrow N$ is the natural projection map, then $(\overline{f},p)$ is clearly
an element of $\mbox{kernel}(\psi)$. Conversely, if $(\alpha,\beta)\in \mbox{kernel}(\psi)$, then
$\beta=p$ and $\alpha=\overline{g}$ for some $g\in K$. There is thus a bijective correspondence between $K$ and kernel$(\psi)$.
We define the group law on $K$ such that this
bijection is an isomorphism.
\end{proof}


It would be desirable to have an ``easier'' description of $K$.
For this, let us look at the set
\[ \mathcal{D}=\{ f: N \rightarrow \Diffeo(M) \}, \]
equipped with
the following group law:
\[ f*g: N\rightarrow \Diffeo(M), z\mapsto f(z)\circ g(z) \mbox{ } \forall f,g \in \mathcal{D} .\]
It is clear that $K<(\mathcal{D},*)$ and that $K$ contains
those elements of $\mathcal{D}$ that satisfy a certain differentiability condition: for a given $f\in K$,
the diffeomorphisms $f(z)$ should change
``smoothly in $z$'' in order for the corresponding map $\overline{f}$
to be differentiable.
Recall that $\Diffeo(M)$ need not be a differentiable manifold, but that it does have
the structure of a Fr\'echet manifold. In fact, it is an open subset of the Fr\'echet manifold $C^{\infty}(M)$ of smooth
self-maps
of $M$ (see \cite{Leslie}, \cite{Onish}). 
We will show that
\[ K= \{ f\in \mathcal{D} \mid f \mbox{ is Fr\'echet }C^{\infty} \}.\]

Let us start by fixing some notation. Take $g\in C^{\infty}(M)$. Consider
the tangent bundle $\pi:TM \rightarrow M$ and denote its pullback under $g$ by $g^*(TM)$:
\[ g^*(TM)=\{ (y,\epsilon) \mid y\in M, \epsilon \in TM \mbox{ with } \pi(\epsilon)=g(y) \} .\]
Given an open, relatively compact set $U\subset M$, we say that
$TM_{\mid U}$ is trivial if
$\overline{U}$
is contained in the image of a coordinate chart $x$. Then, there is local trivialization mapping $v\in T_yM, y\in \mbox{Im}(x)$ to
$(y,b_1,b_2,\ldots ,b_n)\in \mbox{Im}(x)\times \R^n$ where the real numbers $b_i$ are the coordinates of $v$ relative to the basis of
$T_{y}M$ induced by $x$.
Furthermore, we shall say that  $g^*(TM)_{\mid U}$ is trivial if $x$ can be chosen such that
$g(\mbox{Im}(x))\subset \mbox{Im}(\widetilde{x})$
for some chart $\widetilde{x}$. Again, there is a local trivialization mapping
$(y,\epsilon)\in g^*(TM)$ with $y\in \mbox{Im}(x)$ to
$(y,c_1,c_2,\ldots ,c_n) \in \mbox{Im}(x)\times \R^n$ where the $c_i$ are the coordinates of $\epsilon$ relative to the basis of $T_{g(y)}M$
induced by $\widetilde{x}$.\\

Cover $M$ by finitely many open sets $U_{\alpha}$,
such that each $g^*(TM)_{\mid U_\alpha}$ is trivial. Denote the corresponding charts, analogously to $x$ and $\widetilde{x}$
above, by $x_\alpha$ and
$\widetilde{x_\alpha}$.
We call the finite set of triples
\[ (U_\alpha,x_\alpha, \widetilde{x_\alpha}) \]
a trivializing family for $g:M\rightarrow M$.
By definition of trivializing family, the restriction to one of the $U_\alpha$
of a section $s:M\rightarrow g^*(TM)$,
can be seen as a map
$s^{\alpha}:U_\alpha \rightarrow U_\alpha \times \R^n$.
The first component $U_\alpha \rightarrow U_\alpha$ is just the identity. Using $x_\alpha$, we denote the
second component map $\overline{s^\alpha}:x_\alpha^{-1}(U_\alpha) \rightarrow \R^n$.
By definition, we say that a sequence $(s_n)_{n\in \N}$ converges
to $s$ in the Fr\'echet space of smooth sections of $g^*(TM)$ if
\[
\lim_{n\rightarrow \infty} \frac{\partial^l \overline{s^\alpha_n}}
{\partial k_1 \partial k_2 \ldots \partial k_l}= 
\frac{\partial^l \overline{s^\alpha}} {\partial k_1 \partial k_2 \ldots \partial k_l}
\mbox{ uniformly over } x_\alpha^{-1}(U_\alpha) \]
for all $(U_\alpha,x_\alpha, \widetilde{x_\alpha})$, all $l\in \N$ and all $k_1,k_2,\ldots ,k_l \in \R^n.$
%

\begin{proposition} A map
$f:N\rightarrow C^{\infty}(M)$ is Fr\'echet $C^{\infty}$ if and only if the corresponding map
$\overline{f}:M\times N \rightarrow M, (m,n)\mapsto f(n)(m)$ is $C^{\infty}$.
\end{proposition}

\begin{proof}
Assume first
that $f$ is Fr\'echet $C^{\infty}$. Then,
\[ \begin{array}{cccl} 
j: & M\times N & \rightarrow & M\times C^\infty(M) \\
 & (y,z) & \mapsto & (y,f(z))
 \end{array} \]
is Fr\'echet $C^{\infty}$. So, differentiability of $\overline{f}$ is implied by Fr\'echet differentiability of
\[ \begin{array}{cccl}
i: & M\times C^{\infty}(M)& \rightarrow & M \\
 & (y,g) & \mapsto & g(y).
 \end{array} \]

Choose $(y,g)\in M \times C^{\infty}(M)$, fix a trivializing family for $g:M\rightarrow M$ and
denote $S$ the Fr\'echet space of smooth sections of the pullback bundle $g^*(TM)$.
Take an open neighbourhood $U=U_1\times U_2\ni (y,g)$ such that
$(U_1,x_1,\widetilde{x_1})$ is inside the chosen compactifying family.
Denote $x_1^{-1}(U_1)=O$, $\mbox{Im}(\widetilde{x_1})=W$, $C^{\infty}(M)=\{ f: M \rightarrow M \mid f \mbox{ is } C^\infty\}$
and let $x_2:\mathcal{O} \rightarrow C^{\infty}(M), \mathcal{O}\subset S$ be a chart with image $U_2$.
Define
\[ \begin{array}{cccl}
x: & O\times \mathcal{O} & \rightarrow & M\times C^{\infty}(M) \\
 & (o,v)& \mapsto & (x_1(o),x_2(v)).
 \end{array}
 \]
Using the structure of $C^{\infty}(M)$ as a Fr\'echet manifold, we can assume that
$i\circ x:(o,v) \mapsto \exp (\pi(v(x_1(o))))$ with $\pi$ the natural projection of $g^*(TM)$ onto $TM$.
Now, Fr\'echet differentiability of $i$ on $U_1\times U_2$ is equivalent with Fr\'echet differentiability of
\[
\begin{array}{cccl}
\widetilde{i}:= \widetilde{x_1}^{-1} \circ i \circ x: & O\times \mathcal{O} &\rightarrow & \widetilde{x_1}^{-1}(W) \\
 & (o,v) & \rightarrow & \widetilde{x_1}^{-1}(\exp (\pi(v(x_1(o))))),
 \end{array}
 \]
on $O\times \mathcal{O}$ (where we can assume without loss of generality that $U_2$ is small enough for $\widetilde{i}$ to be defined).
Since $\exp:TM \rightarrow M$ is smooth, it suffices to prove that
\[
\begin{array}{cccl}
\gamma: & O\times \mathcal{O} & \rightarrow & TM_{\mid \mbox{Im}(\widetilde{x_1})} \cong \mbox{Im}(\widetilde{x_1})\times \R^n \\
& (o,v) &\mapsto & \pi\circ v(x_1(o))
\end{array}
\]
is Fr\'echet differentiable on $O\times \mathcal{O}$. By differentiability of $g$,
we only need to prove Fr\'echet differentiability for the second component map
$\gamma_2: (o,v)\mapsto \overline{v^1}(o)$.
It is an easy exercise to prove by induction on $l$ that the $l^{\mbox{th}}$ differential $D^l \gamma_2$
exists and that it is given by
\[
\begin{array}{cccl}
D^l\gamma_2 :& (O \times \mathcal{O})\times (\R^n \times S)^l & \rightarrow & \R^n \\
 & (o,s,k_1,h_1,k_2,h_2,\ldots ,k_l,h_l) & \mapsto &
\frac{\partial^l \overline{s^1}}{\partial k_1 \partial k_2 \ldots \partial k_l} (o) +
\sum_{j=1}^l \frac{\partial^{l-1} \overline{h_j^1}}{\partial k_1 \partial k_2 \ldots \hat{\partial k_j} \ldots \partial k_l} (o).
\end{array} \]
Continuity of the differentials of $\gamma_2$ then follows automatically and so we have proven the forward claim of the proposition.\\

To prove the converse, choose $z\in N^k$, denote $f(z)=g$ and for some chart $\overline{x}$ of $N$,
let $V\subset \mbox{Im}(\overline{x})$ be a neighbourhood of $z$ in $N$.
Since $M$ is compact, we can choose $V$ such that the map
\[
\begin{array}{cccl}
v: & M\times \overline{x}^{-1}(V)& \rightarrow & g^*(TM) \\
 & (y,w) & \mapsto & (y, (\exp_{g(y)})^{-1} \overline{f}(y,\overline{x}(w)))
 \end{array}
 \]
is well-defined in the sense that for all $(y,w)\in M\times \overline{x}^{-1}(V)$
there is a totally normal neighbourhood containing
$g(y)$ and $\overline{f}(y,\overline{x}(w))$. The differentiability of $\overline{f}$ clearly implies that of $v$.
It suffices to prove
Fr\'echet differentiability of
\[
\begin{array}{cccl}
\widetilde{f}: & \overline{x}^{-1}(V) & \rightarrow & S\\
& w & \mapsto & v(\cdot,w).
\end{array} \]
%
Fix a compactifying family $(U_\alpha,x_\alpha, \widetilde{x_\alpha})_{\alpha \in \mathcal{A}}$ for $g:M\rightarrow M$
where $\mathcal{A}$ is some index set.
We claim that the $l^{\mbox{th}}$
differential $D^l(\widetilde{f})$ maps $(w',h_1,h_2,\ldots ,h_l)\in \overline{x}^{-1}(V)\times (\R^k)^l$ to the section
$s$ such that
\[ 
\begin{array}{cllll}
\overline{s^{\alpha}}: & x_\alpha^{-1}(U_\alpha) &  \rightarrow & \R^n & \\
& o' & \mapsto & \frac{\partial^l \overline{v(\cdot,w)^{\alpha}}(o)}{\partial h_1 \partial h_2 \ldots \partial h_l} (o',w'),& \forall \alpha \in \mathcal{A}.
 \end{array}
 \]
The claim would imply continuity of $D^l(\widetilde{f})$. Further, in order for the $\overline{s^\alpha}$ to determine a section,
we need to show that for $\alpha,\beta \in \mathcal{A},
y\in U_\alpha \cap U_\beta$, the vector in $T_{g(y)}M$ with coordinates
$\frac{\partial^l \overline{v(\cdot,w)^{\alpha}}(o)}{\partial h_1 \partial h_2 \ldots \partial h_l} (x_\alpha^{-1}(y),w')$
relative to the basis induced by $\widetilde{x_\alpha}$ is the same as the vector with coordinates
$\frac{\partial^l \overline{v(\cdot,w)^{\beta}}(o)}{\partial h_1 \partial h_2 \ldots \partial h_l} (x_\beta^{-1}(y),w')$
relative to the basis induced by $\widetilde{x_\beta}$. To this end, let $A$ be the change of base matrix
from
the basis of $T_{g(y)}M$ induced by $\widetilde{x_\alpha}$ to
the one induced by
$\widetilde{x_\beta}$. It is clear by definition that
\[ A(\overline{v(\cdot,w)^{\alpha}}(x_\alpha^{-1}(y)))=\overline{v(\cdot,w)^\beta}(x_\beta^{-1}(y)), \forall w \in \overline{x}^{-1}(V).\]
We obtain the desired equality since $D^l$ only involves partial derivatives in the second coordinates of $v$.

It remains to prove the claim. By induction, assume the hypothesis is true for some natural number $l$, let us prove it for $l+1$.
We choose $(U_\alpha,x_\alpha,\widetilde{x_\alpha})$ inside our trivializing family.
Let $j\in \N$ and $u_1,u_2,\ldots ,u_j \in \R^n$.
We need to prove that
\[
\frac{1}{t} \frac{\partial^j \overline{D^l(\widetilde{f})(w'+th_{l+1},h_1,h_2,\ldots ,h_l)^{\alpha}}-
\partial^j \overline{D^l(\widetilde{f})(w',h_1,h_2,\ldots ,h_l)^{\alpha}}}
{\partial u_1 \partial u_2 \ldots \partial u_j} \]
converges uniformly for $t\rightarrow 0$ over $x_\alpha^{-1}(U_\alpha)$ to
\[ \frac{\partial^{l+1+j}\overline{v(\cdot,w)^{\alpha}}(y)}{\partial h_1 \partial h_2 \ldots \partial h_{l+1}
\partial u_1 \partial u_2 \ldots \partial u_j} . \]
Pointwise convergence is immediate by differentiability of $v$. Uniform convergence follows from the lemma below.
\end{proof}

\begin{lm} Given $n,k,d \in \N, h\in \R^k$, a $C^1$-map $v:\R^n \times \R^k \rightarrow \R^d$, and
a compact subset $K\subset \R^n$, then
\[ \lim_{t\rightarrow 0} \frac{1}{t}(v(x,y+th)-v(x,y)) = \frac{\partial v}{\partial h}(x,y) \]
uniformly over $x\in K$.
\end{lm}

\begin{proof}
Without loss of generality we can assume that $d=1$. Assume, by contradiction, that the convergence is not uniform over $K$. Then,
\[ \exists \epsilon > 0  \ \forall N\in \N \ \exists t_N<\frac{1}{N}\ \exists x_N\in K \mbox{ such that }
\lvert \frac{1}{t_N}(v(x_N,y+t_Nh)-v(x_N,y)) - \frac{\partial v}{\partial h}(x_N,y) \rvert \geq \epsilon .\]
Consequently,
\[ \exists \epsilon > 0 \ \forall N\in \N \ \exists t_N'<\frac{1}{N} \ \exists x_N\in K \mbox{ such that }
\lvert \frac{\partial v}{\partial h}(x_N,y+t_N'h)- \frac{\partial v}{\partial h}(x_N,y) \rvert \geq \epsilon .\]
Since $K$ is compact, continuity of $\frac{\partial v}{\partial h}(x,y)$ gives us a contradiction.
\end{proof}

We obtain the following
\begin{theorem}
We have the following short exact sequence:
\[ 1 \rightarrow K \hookrightarrow \fvdc(M\times N) \stackrel{\psi}{\rightarrow}
\Diff(N) \rightarrow 1 \] 
with $\psi$ as in corollary \ref{cor:ses}, $K\cong \{ f:N\rightarrow \Diffeo(M) \mid f \mbox{ is Fr\'echet differentiable} \}$.
\end{theorem}
\section{Properly discontinuous actions \label{sc:corollaries}}
\subsection{The Bieberbach theorems}

A group $\Gamma$ acts {\em properly discontinuously} on a space $X$ if the set
\[ \{\gamma \in \Gamma \mid \gamma K \cap K \neq \phi \} \]
is finite for any compact $K\subset X$.
A {\em $k$-dimensional crystallographic group} is a group acting isometrically,
properly discontinuously and cocompactly on $\R^k$. Its structure and some of its properties are described by
the three famous Bieberbach theorems (see \cite{BB1},
\cite{BB2}, \cite{BB3}). Let us recall what they are.
\begin{bieberbach} {\rm Let $\Gamma\subset \R^k \rtimes O(k)=\Iso(\R^k)$ be a $k$-dimensional crystallographic group. Then
$\Gamma$ contains a finite index subgroup $\Gamma^*=\Gamma \cap \R^k$ which is a uniform lattice,
i.e.\! a discrete cocompact subgroup
of $\R^k$.}
\end{bieberbach}
\begin{bieberbach} {\rm Let $\Gamma_1, \Gamma_2 \subset \R^k \rtimes O(k)$ be two $k$-dimensional crystallographic groups.
If $\Gamma_1$ and $\Gamma_2$ are
isomorphic, then they are conjugated by an element of $\mbox{Aff}(\R^k)=\R^k\rtimes GL(k,\R)$.}
\end{bieberbach}
\begin{bieberbach} {\rm Up to isomorphism, there are only finitely many $k$-dimensional crystallographic groups.}
\end{bieberbach}

All three Bieberbach theorems have been generalized to the case of almost-crystallographic groups.
\begin{definition} An {\em almost-crystallographic group} is a group
that acts properly discontinuously, cocompactly and isometrically on
a simply connected, connected, nilpotent Lie
group $N$ that is equipped with a left-invariant metric.\label{def:almcrg}
\end{definition}

The left-invariant metric on $N$ is determined by the choice of an inner product on the Lie algebra $\eta$ of $N$.
Then, $\Iso(N)=N\rtimes C$ where $C$ is the group of automorphisms of $N$ whose differential at the identity preserves the
chosen inner product on $\eta$ (see \cite{wilson}).\\
In $1960$, Auslander generalized the first Bieberbach theorem to almost-crystallographic groups.
{\bieberbach [(Generalization first Bieberbach theorem, Auslander, \cite{Auslander})] Let $\Gamma \subset N\rtimes C$
be an almost-crystallographic group. Then $\Gamma$
contains a finite index subgroup $\Gamma^*=\Gamma\cap N$ which is a uniform lattice of $N$.\label{bb:Auslander}}

It turns out that the first Bieberbach theorem can be generalized in our setting.
\begin{theorem}
Let $M$ be a closed connected Riemannian manifold and let $N$ be a simply
connected, connected, nilpotent Lie group equipped with a left invariant metric.
If $\Gamma$ is a group acting properly discontinuously, cocompactly and isometrically on $M\times N$,
then $\Gamma$ contains a finite index subgroup isomorphic to a uniform lattice of $N$. \label{th:gb}
\end{theorem}
\begin{proof}
Since $N$ is contractible, we have that $(M,N)$ satisfies the \aste-condition.
Corollary \ref{cor:split} thus implies that $\Iso(M\times N)=\Iso(M)\times \Iso(N)$.
Denote
\[ \psi :\Iso(M\times N) \rightarrow \Iso(N) \]
the canonical projection. Let
$\overline{\Gamma}=\psi(\Gamma)$ and let $\Gamma_1$ be the kernel of $\psi_{\mid \Gamma}$. We obtain the following
short exact sequence:
\[ 1\rightarrow \Gamma_1 \rightarrow \Gamma \rightarrow \overline{\Gamma} \rightarrow 1 .\]
Since $\Gamma$ acts properly discontinuously and since $\Gamma_1\subset \Gamma$
maps $M \times \{1\}$ to itself, we have that $\Gamma_1$ is finite.
Clearly, $\overline{\Gamma}$ is an almost-crystallographic group. Theorem
\ref{bb:Auslander} then shows that $\overline{\Gamma}$ contains a finite index
subgroup isomorphic to a uniform lattice of $N$. It is thus virtually-(finitely generated and nilpotent).
Hence, it is poly-(cyclic or finite). In total, we have that
$\Gamma$ is poly-(cyclic or finite) and therefore poly-$\Z$-by-finite. We obtain the following short exact sequence:
\[ 1 \rightarrow P\Z \rightarrow \Gamma \rightarrow F \rightarrow 1, \]
where $P\Z$ is a poly-$\Z$ group and $F$ is a finite group.

The restriction of $\psi$ to the P$\Z-$subgroup is injective since poly-$\Z$-groups are torsion-free.
Then, $P\Z$ is isomorphic to a finite index subgroup of
the almost-crystallographic group $\overline{\Gamma}$. Thus, it is itself an almost-crystallographic group
with a finite index subgroup isomorphic to a uniform lattice of $N$. We conclude that $\Gamma$ contains a finite
index subgroup isomorphic to a uniform lattice of $N$.
\end{proof}
\begin{remark}
{\rm The main tool in proving Theorem \ref{th:gb} is Corollary \ref{cor:split}. Since $N$ is locally compact with a
left-invariant metric, it is complete. Then, de Rham decomposition implies that $\Iso(M\times N)=\Iso(M) \times \Iso(N)$ and therefore
gives an alternate proof.} \label{rem:new}
\end{remark}

We recall that two isomorphic groups of isometries, acting freely, properly discontinuously and cocompactly on $\R$, are
conjugated by an element of $\mbox{Aff}(\R)=\R \rtimes \mbox{GL}_1(\R)$. It is also true that two finite isomorphic groups
acting freely and isometrically on $S^1$ are equal. The following example implies that there is no similar rigidity for $S^1\times \R$.
More concretely, we find two isomorphic groups acting properly discontinuously, cocompactly and isometrically on $S^1\times \R$
such that the induced actions on $S^1$ and $\R$ are free, but these groups cannot be conjugated by an
element of $\Diffeo(S^1)\times \Diffeo(\R)$.

\begin{ex} {\rm Consider $S^1=\{e^{i\theta} \mid \theta \in \R \}$. Choose $\theta_1,\theta_2 \in \R\cmd \Q$ such that
$\theta_1\pm \theta_2\notin \Z$.
Let $\Gamma \subset \Iso(M\times N)$ be the group generated by $(\alpha_1,\alpha_2)$
where $\alpha_1:S^1 \rightarrow S^1$ is multiplication by $e^{2\pi i \theta_1}$ and
$\alpha_2: \R \rightarrow \R$, $x \mapsto x+1$. Analogously, let $\widetilde{\Gamma}$ be the group generated by
$(\beta_1,\beta_2)$ where $\beta_1:S^1\rightarrow S^1$ is multiplication by $e^{2\pi i\theta_2}$
and where $\beta_2=\alpha_2$. Clearly, both groups are infinite cyclic
and they act isometrically, properly discontinuously and cocompactly on $S^1\times \R$. Also, the induced actions
on $S^1$ and $\R$ are free. However, with little effort one can show that
$\langle \alpha_1 \rangle $ and $\langle \beta_1 \rangle$ are not conjugated by a diffeomorphism of $S^1$.}
\end{ex}

The third Bieberbach theorem does not generalize either.
There are infinitely many non-isomorphic groups acting isometrically, properly discontinuously and cocompactly on $S^1\times \{1\}$.
%

\subsection{Talelli's Conjecture}
\setcounter{theorem}{0}
Let us begin by recalling the definition of cohomological dimension.
\begin{definition}
The {\em cohomological dimension} of a group $\Gamma$ is defined by
\[ cd(\Gamma)=\sup \{n \mid H^n(\Gamma;M)\neq 0 \mbox{ for some } \Z\Gamma \mbox{-module } M \}. \]
\end{definition}
There are two definitions in literature for {\em periodic cohomology} of a group. We use the following
\begin{definition} A group $\Gamma$ has {\em periodic cohomology after $k$ steps} if there exists an integer $q>0$ such
that $H^i(\Gamma,-)$ and $H^{i+q}(\Gamma,-)$ are naturally isomorphic functors for all $i>k$.
\end{definition}
In $2005$, Talelli stated the following (Conjecture III of \cite{tal})
\begin{conjecture}
[\em (Talelli, 2005)] {\rm A torsion-free group $\Gamma$ that has periodic cohomology after some steps
has finite cohomological dimension.}
\end{conjecture}

By a result of Mislin and Talelli (\cite{tal:2})
we know that this conjecture holds for the large class of $LH\mathcal{F}$-groups (see \cite{Kropholler}).
Among others, this class contains all linear and all elementary amenable groups.

In $2001$, Adem and Smith have proven that a countable group acts
freely, properly discontinuously and smoothly on some $S^n \times \R^k$ if and only if it has periodic cohomology.
Actually, they use the other definition of periodic cohomology which states
that the isomorphisms of cohomological functors are induced by a cup product map
(see \cite{Adem} for more details).
For the large class of $H\mathcal{F}$-groups it is known that these definitions are equivalent.
Furthermore, it has been
conjectured by Talelli that they are equivalent for all groups.
The Adem-Smith Theorem suggests the following slightly weaker reformulation of the Talelli conjecture.
\begin{conjecture}
[(Talelli reformulated, 2005)] 
{\rm  If $\Gamma$ is a torsion-free group that acts smoothly and
properly discontinuously on $S^n \times \R^k$, then $cd(\Gamma) \leq k$.}
\end{conjecture}

Now, let us replace $S^n$ by any closed, connected Riemannian manifold $M$ and replace $\R^k$
by any $k$-dimensional contractible Riemannian manifold $N$. We obtain the following generalization.
{\conjecture [(Petrosyan, 2007)]
{\rm If $\Gamma$ is a
torsion-free group acting smoothly and properly discontinuously on $M \times N$, then $cd(\Gamma)\leq \dim(N)$.}}\\

In \cite{Petrosyan}, Petrosyan has verified this conjecture in the case of $H\mathcal{F}$-group and when $N$ is $1$-dimensional.
We prove the following

\begin{theorem} Let $\Gamma$ be a torsion-free group that acts properly discontinuously on $M\times N$
where $M$ is closed and connected and where $N$ is contractible. If
each $\gamma \in \Gamma$ acts as a fiberwise volume decreasing map,
then $\Gamma$ acts freely and properly discontinuously on $N$. In particular, $cd(\Gamma)\leq \dim(N)$.
\end{theorem}
\begin{proof}
Let $y_0\in M$ and consider the map
\[ \begin{array}{cccl}
\psi: & \fvdc(M\times N) &\rightarrow & \Diff(N)\\
 & (\alpha,\beta) &\mapsto & \widetilde{\beta} ,
 \end{array} \]
where
\[ \begin{array}{cccl}
\widetilde{\beta}: & N & \rightarrow & N \\
&  z& \mapsto & \beta(y_0,z).
\end{array}
\]
By Corollary \ref{cor:ses}, we have that $\psi$ is a well-defined
epimorphism.

This gives us the following short exact sequence
\[ 1\rightarrow \Gamma_1 \rightarrow \Gamma \rightarrow \overline{\Gamma} \rightarrow 1, \]
where $\overline{\Gamma}=\psi(\Gamma)$ and $\Gamma_1$ is the kernel of $\psi_{\mid \Gamma}$.
Let $z\in N$ and observe that every element of $\Gamma_1$ maps
$M\times \{z\}$ onto itself.
Since $\Gamma$ acts properly discontinuously on $M \times N$ we have that $\Gamma_1$ is finite. Since $\Gamma$
is torsion-free,  $\Gamma_1$ must be trivial and therefore, $\Gamma \cong \overline{\Gamma}$.

Now, $\overline{\Gamma}$ acts freely, smoothly and properly discontinuously on $N$. Since $N$ is contractible,
we have $cd(\Gamma)\leq \dim(N)$.
\end{proof}


\end{document}